\documentclass[11pt]{article}

\usepackage{amsmath,amssymb,amsfonts,textcomp,amsthm,xifthen,psfrag,graphicx,color,MnSymbol}
\usepackage{subcaption,enumitem}
\usepackage[T1]{fontenc}
\usepackage{lscape}

\usepackage{hyperref}

\date{}

\newtheorem{theorem}{Theorem}
\newtheorem{lemma}[theorem]{Lemma}

\newtheorem{prop}[theorem]{Proposition}

\theoremstyle{definition} 

\def\ignore#1{}

\newcommand{\R}{\mathbb{R}}
\newcommand{\Rt}{\mathbb{R}^3}
\newcommand{\Rtt}{\mathbb{R}^{3\times 3}}

\renewcommand{\SS}{\mathbb{S}}

\newcommand{\cF}{\mathcal{F}}
\newcommand{\cC}{\mathcal{C}}

\newcommand{\cS}{\mathcal{S}}

\newcommand{\wtt}{\wtilde t}

\newcommand{\du}{{\delta\!u}}
\newcommand{\deta}{{\delta\!\eta}}

\newcommand{\dsigma}{{\delta\!\sigma}}
\newcommand{\dtau}{{\delta\!\tau}}

\newcommand{\wtilde}{\widetilde}
\newcommand{\<}{\langle{}}
\renewcommand{\>}{\rangle}

\newcommand{\dual}[2]{\<#1\hspace*{.5mm},#2\>}
\newcommand{\vdual}[2]{(#1\hspace*{.5mm},#2)}

\DeclareMathOperator{\supp}{supp}
\DeclareMathOperator{\diam}{diam}
\DeclareMathOperator{\dist}{dist}

\DeclareMathOperator{\dev}{dev}
\DeclareMathOperator{\tr}{tr}
\DeclareMathOperator{\id}{\mathrm{id}}
\newcommand{\CC}{\mathbb{C}}
\renewcommand{\AA}{\mathbb{A}}

\newcommand{\grad}{\nabla}

\DeclareMathOperator{\curl}{curl}

\newcommand{\strain}[1]{\varepsilon(#1)}

\renewcommand{\div}{\operatorname{div}}

\newcommand{\divref}{\widehat{\operatorname{div}}}

\newcommand{\PiolaK}[1]{\mathcal{PK}_{{#1}}}
\newcommand{\PiolaKc}[1]{\mathcal{PK}^\perp_{{#1}}}
\newcommand{\PiolaCurl}[1]{\mathcal{P}_{#1}}
\newcommand{\Inn}[1]{\mathcal{I}^{nn}_{#1}}
\newcommand{\XX}[1]{\mathbb{X}(#1)}

\newcommand{\sreg}{{s'}}

\newcommand{\mesh}{\mathcal{T}}
\newcommand{\el}{T}
\newcommand{\elref}{{\widehat T}}
\newcommand{\xref}{{\widehat x}}
\newcommand{\yref}{{\widehat y}}
\newcommand{\zref}{{\widehat z}}
\newcommand{\Fref}{{\widehat F}}
\newcommand{\nref}{{\widehat n}}
\newcommand{\vref}{{\widehat v}}
\newcommand{\tauref}{{\widehat\tau}}
\newcommand{\sigmaref}{{\widehat\sigma}}

\newcommand{\trnn}{{\gamma_{1,\mathit{nn}}}}

\newcommand{\G}[1]{{\Gamma_\mathit{#1}}}
\newcommand{\g}[1]{{g_{#1}}}

\title{Normal-normal continuous symmetric stress approximation in three-dimensional
       linear elasticity
\thanks{Supported by ANID-Chile through FONDECYT project 1230013}}
\author{
Carsten Carstensen\thanks{
Department of Mathematics, Humboldt--Universit\"at zu Berlin,
Unter den Linden 6, 10099 Berlin, Germany,
email: {\tt cc@math.hu-berlin.de}}
\and
Norbert Heuer\thanks{
Facultad de Matem\'aticas, Pontificia Universidad Cat\'olica de Chile,
Avenida Vicu\~na Mackenna 4860, Santiago, Chile,
email: {\tt nheuer@uc.cl}}}

\begin{document}
\maketitle
\begin{abstract}
We present a conforming setting for a mixed formulation of linear elasticity
with symmetric stress that has normal-normal continuous components across
faces of tetrahedral meshes. We provide a stress element for this formulation
with $30$ degrees of freedom that correspond to standard boundary conditions.
The resulting scheme converges quasi-optimally and is locking free.
Numerical experiments illustrate the performance.

\end{abstract}

\section{Introduction and model problem}

Adams \& Cockburn \cite{AdamsC_05_MFE} and Arnold \emph{et al.} \cite{ArnoldAW_08_FES}
presented and analyzed, with slightly different degrees of freedom,
the first stable element for conforming stress approximations
in 3D linear elasticity on tetrahedral meshes.
It provides pointwise symmetric piecewise polynomial approximations
of degree at least four with continuous normal components across faces.
With its 162 degrees of freedom (plus displacement degrees) this element is prohibitively
expensive for practical applications. Modifications and generalizations to
arbitrary dimension and higher order have been proposed by Hu \& Zhang, Hu
\cite{HuZ_15_FSM,Hu_15_FEA} and Chen \& Huang \cite{ChenH_22_FEDb}.
The difficulty of constructing conforming elements
with few degrees of freedom has been the reason for alternative developments,
e.g., of discontinuous and non-conforming methods with weak symmetry.
But the quest for cheap stable and conforming stress elements goes on,
in particular on tetrahedra which provide important geometric flexibility.
Hu and Zhang \cite{HuZ_16_FEA} proposed to use low-degree polynomial
tensors enriched with edge and face related bubble polynomials to achieve
stability, in this way reducing the degrees of freedom to $48$.
A variation of this approach and extension to arbitrary dimension is given in \cite{HuangZZZ_24_NLO}.

In this paper we present a pointwise symmetric tetrahedral element of quadratic
polynomials with $30$ degrees of freedom (they can be reduced to $12$ by static condensation)
and provide a variational framework that renders the element conforming.
In geometrical terms our element leads to linear normal-normal components on faces
that are continuous across elements. This approach goes back to the seminal TDNNS
---tangential displacement and normal-normal stress continuous---
method by Pechstein \& Sch\"oberl
\cite{PechsteinS_11_TDN,PechsteinS_12_AMF,PechsteinS_18_ATM} that in turn
transfers the Hellan--Herrmann--Johnson idea from plate bending to linear elasticity,
cf.~\cite{Hellan_67_AEP,Herrmann_67_FEB,Johnson_73_CMF}.
In \cite{CarstensenH_NNC} we presented a continuous and discrete framework
for this setting in the case of plane elasticity with mixed approximation on triangular
meshes. Its advantages are 

\begin{itemize} \setlength\itemsep{0em}
\item a low number of degrees of freedom related to
\item physical degrees of freedom that allow for the
\item treatment of any boundary condition that is physically relevant, and a
\item variational framework that makes the element conforming.
\end{itemize}
Furthermore, it is
\begin{itemize} \setlength\itemsep{0em}
\item provably stable and locking free.
\end{itemize}

This paper solves the critical open case of three-dimensional elasticity
on tetrahedral elements. We stress the fact that all the advantages listed above remain
valid in three dimensions. We expect that the existence of a conforming
framework facilitates, for instance, the analysis of a posteriori error estimation.
Perhaps more important is the fact that, differently from the previously mentioned
$H(\mathrm{div})$-conforming methods, the degrees of freedom of our element avoid
vertex values and edge moments of stresses. Such degrees do not relate to
boundary conditions in engineering applications and are only well posed
when making artificially high regularity assumptions.

In order not to be repetitive, we only briefly recall settings and cite proofs from
\cite{CarstensenH_NNC}, and also refer to \cite{Heuer_GMP} for general results
on a hybrid framework. Instead, we focus on the new ingredients required
for the three-dimensional setting. In particular, we restrict the analysis to
homogeneous Dirichlet boundary conditions.
The incorporation and analysis of general boundary conditions are similar
to the two-dimensional situation treated in \cite{CarstensenH_NNC}.

{\bf Model problem.}
We consider a bounded, polyhedral Lipschitz domain $\Omega\subset\Rt$
with boundary $\Gamma:=\partial\Omega$ decomposed into non-intersecting relatively open pieces
\begin{align*} 
   \Gamma=\mathrm{closure}(\G{hc}\cup\G{sc}\cup\G{ss}\cup\G{sf}).
\end{align*}
Boundary conditions are of hard clamped (``$\mathit{hc}$''), soft clamped (``$\mathit{sc}$''),
simply supported (``$\mathit{ss}$'') and traction (``$\mathit{sf}$'' for stress free) types
on the corresponding boundary pieces. Individual sets are either empty or unions
of connected sets of positive measure, subject to the validity of Korn's inequality.
For the discretization it has to be assumed that the boundary pieces are conforming
with respect to the faces of the tetrahedral elements, that is, interfaces between
different boundary pieces do not cut faces.

The model problem of linear elasticity with general boundary conditions reads
\begin{align*} 
   \sigma=2\mu\strain{u}+\lambda(\div u)\id
   \quad\text{and}\quad -\div\sigma=f\quad\text{in}\ \Omega,
\end{align*}
\begin{equation} \label{BC}
\begin{alignedat}{4}
   u&=\g{D} &&\text{on}\ \G{hc},\qquad
   &&u\cdot n=\g{D,n},\ \pi_t(\sigma n)=\g{N,t}\quad &&\text{on}\ \G{sc},
   \\
   \pi_t u&=\g{D,t},\ n\cdot\sigma n=\g{N,n}\quad &&\text{on}\ \G{ss},
   &&\sigma n=\g{N} &&\text{on}\ \G{sf}.
\end{alignedat}
\end{equation}
Here, $\strain{u}$ and $\id$ are the strain and identity tensors, respectively,
and $\mu,\lambda>0$ are the Lam\'e parameters.
We assume that $\mu$ is fixed but allow for general $\lambda$, corresponding to
an unrestricted Poisson ratio $\nu\in (0,1/2)$.
For short, we write $\sigma=\CC\strain{u}$ with elasticity tensor $\CC$, and
let $\AA:=\CC^{-1}$ be the compliance tensor. Furthermore,
$n$ is the exterior unit normal vector on $\Gamma$, and
$\pi_t u:=(u-(u\cdot n)n)|_\Gamma$ is the tangential trace on $\Gamma$ of vector fields $u$.
Boundary data have to be consistent with the existence of a function
$u\in H^1(\Omega;\Rt)$ and tensor $\sigma\in H(\div,\Omega;\R^{3\times 3})$
that satisfy \eqref{BC} as appropriate traces.
Note that $f,\g{D},\g{D,t},\g{N},\g{N,t}$ are vector-valued and
$\g{D,n},\g{N,n}$ are scalar functions.

As mentioned before, our forthcoming variational formulation and mixed finite element
discretization are compatible with the general boundary conditions \eqref{BC}.
Only for ease of presentation will we restrict ourselves to homogeneous Dirichlet
conditions, thus consider
\begin{align} \label{model}
   \sigma=\CC\strain{u},\quad
   -\div\sigma=f\quad\text{in}\ \Omega,\qquad
   u|_\Gamma=0,
\end{align}
and refer to \cite{CarstensenH_NNC} for details concerning general boundary conditions.

{\bf Overview.} The remainder is structured as follows. In Section~\ref{sec_trace}
we introduce our notation for spaces and norms. At the center is the definition
of skeleton trace spaces of tangential and normal displacements,
their characterization (Propositions~\ref{prop_trace_norm},~\ref{prop_trace_tang})
and density in the space of displacement traces (Proposition~\ref{prop_cont}).
The proof of the density result is based on a 3D version of
Tartar's result that says that the set of $C^\infty$-functions with compact support
that vanish in a neighborhood of a point $x\in\Omega\subset\R^2$ are dense in $H^1_0(\Omega)$.
This is Proposition~\ref{prop_Tartar}. Having these results at hand,
we follow the canonical procedure from \cite{CarstensenH_NNC} to define
the space $H_{nn}(\div,\mesh;\SS)$ of symmetric piecewise $H(\mathrm{div})$-tensors
with normal-normal continuous traces on polyhedral meshes and a related displacement
trace operator.
Proposition~\ref{prop_conf} summarizes corresponding conformity and norm relations.
We note that not all the results on the trace spaces are required to prove
Proposition~\ref{prop_conf}.
We present them for their general relevance and possible future reference.
Section~\ref{sec_mixed} presents the mixed variational formulation of the model
problem that defines the stress as an element of $H_{nn}(\div,\mesh;\SS)$.
Theorem~\ref{thm_VF} establishes its Poisson-ratio robust well-posedness.
Section~\ref{sec_fem} is devoted to the discrete setting and analysis.
In \S\ref{sec_el} we introduce the stress element and disclose its degrees of freedom.
The corresponding conforming piecewise polynomial approximation space
$P^2_{nn}(\mesh)$ is the subject of \S\ref{sec_approx}, with approximation
properties shown by Proposition~\ref{prop_Inn}.
The resulting mixed finite element scheme is presented in \S\ref{sec_fem_conv},
and shown to be locking free with expected convergence orders by
Theorem~\ref{thm_Cea}. In Section~\ref{sec_num} we present a numerical example
that illustrates the locking-free convergence of our method.

Throughout, $a\lesssim b$ stands for $a\le Cb$ with a generic constant $C>0$
that is independent of involved functions and mesh $\mesh$.
Relation $a\gtrsim b$ means that $b\lesssim a$.

\section{Analytical framework} \label{sec_trace}

In this section we introduce and analyze spaces and trace operators that provide the basis
of our mixed formulation. The next subsection presents some notation
and collects results for trace spaces of vector functions.
In \S\ref{sec_trace_nn} we define the stress space of pointwise symmetric tensors
with continuous normal-normal traces and establish duality relations with
a trace space.

\subsection{Notation, spaces and trace operators}

For sub-domains and surfaces $\omega\subset\overline{\Omega}$ we use standard Lebesgue
and Sobolev spaces $L_q(\omega;U)$ ($q>1$) and $H^s(\omega;U)$ ($s>0$) of
$U$-valued functions on $\omega$ with $U\in\{\R,\Rt,\SS\}$.
Here, $\SS\subset\Rtt$ denotes the space of symmetric constant tensors.
The generic $L_2(\omega;U)$-inner product and norm are $\vdual{\cdot}{\cdot}_\omega$
and $\|\cdot\|_\omega=\|\cdot\|_{0,\omega}$, respectively.
For integer $s$, $\|\cdot\|_{s,\omega}$ denotes the standard norm in $H^s(\omega;U)$,
except for $s=1$ and $U=\Rt$ in which case
$\|\cdot\|_{1,\omega}:=\Bigl(\|\cdot\|_\omega^2+\|\strain{\cdot}\|_\omega^2\Bigr)^{1/2}$
with symmetric gradient $\strain{\cdot}:=(\grad(\cdot)+\grad(\cdot)^\top)/2$.
For non-integer $s>0$, $\|\cdot\|_{\omega,s}$ and $|\cdot|_{\omega,s}$ denote the
Sobolev-Slobodeckij norms and seminorms, cf.~\cite{Adams}.
Given $s,r\ge 0$ and $\omega\subset\Omega$, we need the space
\[
   H^{s,r}(\div,\omega;\SS):=\{\tau\in H^s(\omega;\SS);\; \div\tau\in H^r(\omega;\Rt\}
\]
with row-wise application $\div\tau$ of the divergence operator, and denote
$H(\div,\omega;\SS) :=$\linebreak $H^{0,0}(\div,\omega;\SS)$
(identifying $H^0(\Omega;U)$ with $L_2(\Omega;U)$).
The latter space is provided with the graph norm
$\|\cdot\|_{\div,\omega}:=\Bigl(\|\cdot\|_\omega^2+\|\div\cdot\|_\omega^2\Bigr)^{1/2}$.
We also need the topological dual $H^{-1}(\omega)$ of $H^1_0(\omega)$ with norm
$\|\cdot\|_{-1,\omega}$.
We generally drop index $\omega$ in the notation of norms and inner products when $\omega=\Omega$.
The trace of a tensor $\tau$ is $\tr(\tau):=\tau:\id$ with identity tensor $\id$
and Frobenius product ``$:$''; the deviatoric part of $\tau$ reads
$\dev(\tau):=\tau-\tr(\tau)\id/3$.

We consider a regular decomposition $\mesh$ of $\Omega$ into shape-regular tetrahedra
and formally denote by $\cS:=\{\partial\el;\;\el\in\mesh\}$ its skeleton.
We need the piecewise constant function $h_\mesh$ with $h_\mesh|_\el:=\diam(\el)$
for $\el\in\mesh$.
The set of faces of $\el\in\mesh$ is $\cF(\el)$, and the sets of all (resp. interior) faces is
$\cF:=\cup_{\el\in\mesh}\cF(\el)$ (resp. $\cF(\Omega)$). In the following,
$\dual{\cdot}{\cdot}_K$ denotes the generic $L_2$-duality on $K\in\cF\cup\{\Gamma\}$.
Decomposition $\mesh$ gives rise to product spaces, with corresponding notation
where $\Omega$ is replaced with $\mesh$, e.g., $H^s(\mesh;U):=\Pi_{\el\in\mesh} H^s(\el;U)$.
Throughout, we identify functions of product spaces with piecewise defined functions.
Mesh $\mesh$ induces the canonical trace operator
\begin{align*}
   \gamma:\;&\left\{\begin{array}{cll}
               H^1(\Omega;\Rt) &\rightarrow& H(\div,\mesh;\SS)^*,\\
               v &\mapsto& \dual{\gamma(v)}{\tau}_\cS
               := \vdual{\strain{v}}{\tau} + \vdual{v}{\div\tau}_\mesh
            \end{array}\right.
\end{align*}
with support on $\cS$ and trace space
\[
   H^{1/2}(\cS;\Rt) := \gamma(H^1(\Omega;\Rt)).
\]
Here, $\vdual{\cdot}{\cdot}_\mesh$ is the generic notation for
$L_2(\mesh;U)$ dualities, $U\in\{\R,\Rt,\SS\}$.
There is an inherent duality between $H^{1/2}(\cS;\Rt)$ and $H(\div,\mesh;\SS)$,
\begin{align} \label{duality}
   \dual{\varphi}{\tau}_\cS := \dual{\gamma(v)}{\tau}_\cS
   \quad (\varphi\in H^{1/2}(\cS;\Rt), \tau\in H(\div,\mesh;\SS)),
\end{align}
where $v$ is any element of the pre-image $\gamma^{-1}(\varphi)\subset H^1(\Omega;\Rt)$,
$\varphi=\gamma(v)$.

Given the tangential and normal trace operators
\begin{align*}
   \gamma_t:\;&\left\{\begin{array}{cll}
               H^1(\Omega;\Rt) &\rightarrow& H^1(\mesh;\Rt)^*,\\
               v &\mapsto& \dual{\gamma_t(v)}{z}_\cS
               := \vdual{\curl v}{z} - \vdual{v}{\curl z}_\mesh,
            \end{array}\right.\\
   \gamma_n:\;&\left\{\begin{array}{cll}
               H^1(\Omega;\Rt) &\rightarrow& H^1(\mesh;\Rt)^*,\\
               v &\mapsto& \dual{\gamma_n(v)}{z}_\cS
               := \vdual{\div v}{z} - \vdual{v}{\grad z}_\mesh,
            \end{array}\right.
\end{align*}
we introduce their kernels
\(
   H^1_n(\Omega,\cS;\Rt):=\mathrm{ker}(\gamma_t)
\)
and
\(
   H^1_t(\Omega,\cS;\Rt):=\mathrm{ker}(\gamma_n).
\)
These are closed subspaces of $H^1(\Omega;\Rt)$ and give rise to the trace spaces
\begin{align}
   \wtilde H^{1/2}_{0,n}(\cS;\Rt) &:= \gamma(H^1_n(\Omega,\cS;\Rt)\cap H^1_0(\Omega;\Rt)),
   \label{H_trace_nor} \\ \nonumber
   \wtilde H^{1/2}_{0,t}(\cS;\Rt) &:= \gamma(H^1_t(\Omega,\cS;\Rt)\cap H^1_0(\Omega;\Rt))
\end{align}
%
which consist of normal and tangential vector functions on $\cS$.
To be more specific, let us introduce some further notation.
For a function $v\in H^1(\el;U)$ ($\el\in\mesh$, $U\in\{\R,\Rt\}$)
its traces on the boundary $\partial\el$ and faces $F\in\cF(\el)$
are denoted by $v|_{\partial\el}\in H^{1/2}(\partial\el;U)$
and $v|_F\in H^{1/2}(F;U)$, respectively.
Restrictions of functions $\varphi\in H^{1/2}(\cS;\Rt)$ to subsets of $\cS$
are defined through their extensions, e.g., $\varphi|_\Gamma:=v|_\Gamma$
for $v\in H^1(\Omega;\Rt)$ with $\gamma(v)=\varphi$.
We denote by $\wtilde H^{1/2}(F)\subset H^{1/2}(F)$ the restriction onto $F$
of the subspace of functions $v\in H^{1/2}(\partial T)$ that vanish on
$\partial\el\setminus\overline{F}$, $F\in\cF(\el)$, $\el\in\mesh$.

\begin{prop}[$\wtilde H^{1/2}_{0,n}(\cS;\Rt)$] \label{prop_trace_norm}
Any function $\varphi\in H^{1/2}(\cS;\Rt)$ satisfies
$\varphi\in \wtilde H^{1/2}_{0,n}(\cS;\Rt)$ if and only if
$\pi_t\varphi|_F=0$, $\varphi\cdot n|_F\in \wtilde H^{1/2}(F)$ for all $F\in\cF(\Omega)$,
and $\varphi|_\Gamma=0$.
\end{prop}

\begin{proof}
Given $\varphi\in H^{1/2}(\cS;\Rt)$, $\pi_t\varphi|_F=0$ for all $F\in\cF(\Omega)$
and $\varphi|_\Gamma=0$ imply $\varphi\in \wtilde H^{1/2}_{0,n}(\cS;\Rt)$.
To see the other direction consider $v\in H^1(\Omega;\Rt)$ with $\varphi=\gamma(v)$.
For a tetrahedron $\el\in\mesh$ let $x_j$ denote its vertices with opposite faces
$F_j$ ($j=1,\ldots,4$).
The corresponding barycentric coordinates and normal vectors are
$\lambda_j$ and $n_j$. Assuming that the tangential traces of $v$ on faces $F_j$ vanish
we have to show that $v\cdot n_j|_{F_j}\in\wtilde H^{1/2}(F_j)$ ($j=1,\ldots,4$).

It is enough to show that $v\cdot n_1|_{F_1}\in\wtilde H^{1/2}(F_1)$.
We complement vector $n_1$ with two linearly independent
vectors $t_{1,1},t_{1,2}$ that are orthogonal to $n_1$.
There are constants $c_j,d_j,\alpha_j\in \R$, $\alpha_j\not=0$ ($j=2,3,4$) such that
\[
   n_1+c_j t_{1,1}+d_j t_{1,2}=\alpha_j n_{j'}\times n_{j''},
   \quad j=2,3,4,
\]
with $j',j''\in\{2,3,4\}\setminus\{j\}$, $j'<j''$, that is,
$(j,j',j'')$ is a permutation of $(2,3,4)$.
In fact, since $\{n_1,t_{1,2},t_{1,2}\}$ is a basis of $\Rt$,
there are numbers $\xi_j,\eta_j,\rho_j\in\R$ such that
$\xi_j n_1+\eta_j t_{1,1}+\rho_j t_{1,2}=n_{j'}\times n_{j''}$ ($j=2,3,4$).
Note that $\xi_j\not=0$ because $n_1\cdot(n_{j'}\times n_{j''})\not=0$
by the linear independence of $n_1,n_{j'},n_{j''}$.
It follows that $c_j=\eta_j/\xi_j$, $d_j=\rho_j/\xi_j$ and $\alpha_j=1/\xi_j$ ($j=2,3,4$).
We continue to select
\[
   g:= \sum_{j=2}^4 (n_1+c_j t_{1,1}+d_j t_{1,2})\lambda_j
\]
and define
$w:=g\cdot v|_\el\in H^1(\el)$. Since by assumption the tangential traces
of $v$ on $\partial\el$ vanish, $w$ satisfies
\begin{align*}
   w|_{F_1}= \sum_{j=2}^4 \lambda_j(n_1+c_j t_{1,1}+d_j t_{1,2})\cdot v|_{F_1}
           = \sum_{j=2}^4 \lambda_j n_1\cdot v|_{F_1} = n_1\cdot v|_{F_1}
\end{align*}
and, because $\lambda_j|_{F_j}=0$ and $v|_{F_j}=(v\cdot n_j) n_j|_{F_j}$,
\begin{align*}
   w|_{F_m}= \sum_{j=2}^4 \lambda_j(n_1+c_j t_{1,1}+d_j t_{1,2})\cdot v|_{F_m}
           = \sum_{j\in\{2,3,4\}\setminus\{m\}}
             \lambda_j\alpha_j(v\cdot n_m) (n_{j'}\times n_{j''})\cdot n_m|_{F_m}
           = 0
\end{align*}
for $m=2,3,4$ by the orthogonality of $n_m$ and $n_{j'}\times n_{j''}$
($m=j'$ or $m=j''$ in the sum above).
We conclude that $w|_{\partial\el}\in H^{1/2}(\partial\el)$
is a zero extension of $v\cdot n_1|_{F_1}$ to the other faces of $\el$,
that is, $\varphi\cdot n_1|_{F_1}=v\cdot n_1|_{F_1}\in\wtilde H^{1/2}(F_1)$.
\end{proof}

There is no corresponding characterization of $\wtilde H^{1/2}_{0,t}(\cS;\Rt)$
as face-wise ``$\wtilde H^{1/2}$-space''. To see this, consider a tetrahedron $T$ with vertices
$x_i$, faces $F_i$ opposite to $x_i$, normal vectors $n_i$ on $F_i$,
and barycentric coordinates $\lambda_i$, $i=1,\ldots,4$.
Let $e_{1,2}\in\Rt\setminus\{0\}$ be a vector that is generated by the edge shared by
$F_1$ and $F_2$.
Function $v:=\lambda_3 \lambda_4 e_{1,2}$ satisfies $v\in H^1(T;\Rt)$,
$v\cdot n_i|_{F_i}=0$ ($i=1,2$) and $v|_{F_i}=0$ ($i=3,4$), but its tangential component
$\pi_t v|_{F_1}=v|_{F_1}$ does not vanish on the shared edge.
Therefore, $\pi_t v|_{F_1}$ cannot be continuously extended by zero onto $F_2$
in the trace space on $\partial T$,
which would be an appropriate characterization of the ``$\wtilde H^{1/2}(F_1;\Rt)$-property''
that corresponds to the scalar case of normal components in $\wtilde H^{1/2}(F)$ just considered.
However, the edge-normal components of tangential face-traces of
functions from $\wtilde H^{1/2}_{0,t}(\cS;\Rt)$ do have such a property.
A proper definition requires more notation.

For a simplex $\el\in\mesh$ with face $F\in\cF(\el)$ let $n_{\partial F}$ denote the exterior
unit normal vector along $\partial F$ that is normal to $n$ (a vector normal to $F$).
For an edge $e$ of $F$, $n_{F,e}$ is the restriction of $n_{\partial F}$ to $e$.
We define $\wtilde H^{1/2}_{tn}(F;\Rt)$ as the subspace of
functions $\varphi\in H^{1/2}(F;\Rt)$ such that,
for every edge $e$ of $F$ shared with a different face $F'\in\cF(\el)$,
there exists a function $v\in H^{1/2}(\partial\el)$ with
$v|_F=\varphi|_F\cdot n_{F,e}$ and $v|_{F'}=0$. In other words,
face-wise \underline{t}angential traces have edge-\underline{n}ormal components that,
for every edge, can be continuously extended in $H^{1/2}(\partial\el;\Rt)$
by zero across the edge to the neighboring face
(suggesting the ``$tn$'' index notation).

\begin{prop}[$\wtilde H^{1/2}_{0,t}(\cS;\Rt)$] \label{prop_trace_tang}
Any function $\varphi\in H^{1/2}(\cS;\Rt)$ satisfies
$\varphi\in \wtilde H^{1/2}_{0,t}(\cS;\Rt)$ if and only if
$\varphi\cdot n|_F=0$, $\varphi|_F\in \wtilde H^{1/2}_{tn}(F;\Rt)$ for all $F\in\cF(\Omega)$,
and $\varphi|_\Gamma=0$.
\end{prop}

\begin{proof}
The proof is similar to the proof of Proposition~\ref{prop_trace_norm} and
we merely discuss the non-trivial implication.
Given $\varphi\in \wtilde H^{1/2}_{0,t}(\cS;\Rt)$, $v\in H^1(\Omega;\Rt)$ with $\varphi=\gamma(v)$,
and $\el\in\mesh$, it is immediate that $\varphi\cdot n|_{\partial\el}=0$,
and we have to show that $\varphi|_F\in \wtilde H^{1/2}_{tn}(F;\Rt)$ for $F\in\cF(\el)$.
Let $F,F'\in\cF(\el)$ be given with $F\not= F'$ and common edge $e$. For the corresponding
normal vectors $n,n'$, there are $c,\alpha\in\R$ such that $n_{F,e}+c n=\alpha n'$.
In fact, $n,n',n_{F,e}$ are linearly dependent because
they are orthogonal to edge $e$ and $n_{F,e}\cdot n'\not=0$.
It follows that $w:=(n_{F,e}+cn)\cdot v|_\el\in H^1(\el)$.
Furthermore, $w|_F=\varphi\cdot n_{F,e}|_F$ and $w|_{F'}=\alpha n'\cdot\varphi|_{F'}=0$,
that is, $\varphi|_F\in \wtilde H^{1/2}_{tn}(F,\Rt)$ as wanted.
\end{proof}

We generalize the concept of ``a point is too small'' in two dimensions from
Tartar \cite[\S 17]{Tartar_07_ISS} to three dimensions. In two dimensions,
it says that $C_0^\infty(\Omega\setminus\{x\})$ for $x\in\Omega\subset\R^2$
is still dense in $H^1_0(\Omega)$.

Consider a finite set of pairwise distinct, straight lines $L_1,\ldots,L_J$ in $\Rt$,
$L_j=\{p_j+s q_j;\;s\in\R\}$ with $p_j,q_j\in\R^3$, $|q_j|=1$,
and $L_j\cap\Omega\not=\emptyset$, $j=1,\ldots,J$.
(Here, $\Omega\subset\R^3$ can be any bounded open set.)
The set of lines
\(
   L:=L_1\cup\ldots \cup L_J
\)
is ``too small'', in the following sense.

\begin{prop}[density] \label{prop_Tartar}
The set $C_0^\infty(\Omega\setminus L)$ is dense in $H^1_0(\Omega)$.
\end{prop}

\begin{proof}
Without loss of generality we assume that $\diam(\Omega)\le 1$ (otherwise we
scale $\Omega$).
Let $f\in H^1_0(\Omega)$ and $\epsilon>0$ be given.
We approximate $f$ in several steps.

{\bf 1.} Since $C^\infty_0(\Omega)\subset H^1_0(\Omega)$ is dense
we find $f_0\in C^\infty_0(\Omega)$ with $\|f-f_0\|_1<\epsilon/5$.

{\bf 2.} Let us assume that $|f_0|\le 1$ in $\Omega$ (for the other case see the end of the proof).
We approximate separately the non-negative and non-positive parts
$f_+:=\max\{0,f_0\}$ and $f_-:=f_+-f_0$ of $f_0$.
Note that $f_\pm\in H^1_0(\Omega)$, cf.~\cite[Theorem~4.4(iii)]{EvansG_15_MTF},
and $0\le f_\pm\le 1$.

{\bf 3.} The function $\log\log(1/|\cdot|)$ is a standard
example of an unbounded $H^1$-function on sufficiently small
two-dimensional neighborhoods of the origin.
We consider a parameter $0<\delta<1$ to be selected later, and define
\[
   \chi(r):=\Bigl(\log(-e\log\delta)-\log\log(1/r)\Bigr)_+
   \quad\text{for}\quad 0<r<1,
\]
with non-negative part $(\cdot)_+$.
It is elementary that $\chi$ is monotonically increasing in $(0,1)$,
$\chi=0$ in $[0,\delta^e]$, $0<\chi<1$ in $(\delta^e,\delta)$, $\chi(\delta)=1$,
and $\chi>1$ in $(\delta,1)$;
$e$ is Euler's number and $\delta^e$ means $\delta$ raised to the power $e$.
We continue to define
\[
   g_j(x):=\min\{1,\chi(\dist(x,L_j))\} \quad (x\in\Omega, j=1,\ldots,J).
\]
Transformation to cylindrical coordinates along $L_j$ reveals
\[
   \|g_j\|_{1,U_j}\le \sqrt{\pi\delta^2-2\pi/\log\delta}\quad\text{for}\quad
   U_j:=\{x\in\Omega;\;\dist(x,L_j)<\delta\} = \{x\in\Omega;\; g_j(x)<1\}
\]
($j=1,\ldots,J$) where we used that $\diam(\Omega)\le 1$.
We denote
\(
   U:=U_1\cup\ldots\cup U_J=\{x\in\Omega;\; \dist(x,L)<\delta\}
\)
and conclude for $j=1,\ldots,J$ that
\begin{align} \label{loglog_bound}
   \|g_j\|_{1,U}^2 \le \sum_{k=1}^J \|g_j\|_{1,U_k}^2
   \le J(\pi\delta^2-2\pi/\log\delta) + \sum_{k\not=j} |U_k| \to 0\quad
   \text{as}\ \delta\to 0\ \text{for}.
\end{align}

{\bf 4.} We define preliminary approximations of $f_\pm$ by
$g_\pm:=\min\{f_\pm,g_1,\ldots,g_J\}$ on $\Omega$.
They vanish in a neighborhood of $L\cap\Omega$, $\dist(\supp g_\pm, L)>0$,
satisfy $g_\pm\in H^1_0(\Omega)$ and
\[
   g_1(x)=\ldots=g_J(x)=1\ge f_\pm(x)=g_\pm(x)\quad
   \text{at}\ x\in\Omega\ \text{with}\ \dist(x,L)>\delta.
\]
Thus
\(
   \|f_\pm-g_\pm\|_1 = \|f_\pm-g_\pm\|_{1,U}.
\)
%
By the triangle inequality and the coarse bounds
$0\le g_\pm\le f_\pm+g_1+\ldots+g_J$,
$|\grad g_\pm|\le |\grad f_\pm| + |\grad g_1| + \ldots + |\grad g_J|$
(almost everywhere in $U$) we find with \eqref{loglog_bound} that
\begin{align*}
   \frac 12 \|f_\pm-g_\pm\|_{1,U}^2
   &\le \|f_\pm\|_{1,U}^2 + \|f_\pm+g_1+\ldots+g_J\|_U^2
                                 + \||\grad f_\pm| + |\grad g_1| + \ldots |\grad g_J|\|_U^2\\
   &\le \|f_\pm\|_{1,U}^2
            + (J+1)\bigl(\|f_\pm\|_{1,U}^2 + \|g_1\|_{1,U}^2 +\ldots+ \|g_J\|_{1,U}^2\bigr)\\
   &= (J+2)\|f_\pm\|_{1,U}^2 + (J+1)\Bigl(\|g_1\|_{1,U}^2 + \ldots + \|g_J\|_{1,U}^2\Bigr)
   \to 0\quad (\delta\to 0)
\end{align*}
because $f_\pm\in H^1(\Omega)$ and $|U|\to 0$ as $\delta\to 0$.
Therefore, $\|f_\pm-g_\pm\|_1<\epsilon/5$ for some fixed, sufficiently small $\delta>0$.
By standard mollification of $g_\pm$ we find $\wtilde g_\pm\in C^\infty_0(\Omega\setminus L)$ with
$\|g_\pm-\wtilde g_\pm\|_1\le \epsilon/5$.

\medskip
A combination of steps 1--4 shows that $\wtilde g:=\wtilde g_+-\wtilde g_-$ satisfies
$\wtilde g\in C^\infty_0(\Omega\setminus L)$ and
\[
   \|f-\wtilde g\|_1
   \le
   \|f-f_0\|_1 + \|f_+ -g_+\|_1 + \|f_- - g_-\|_1 + \|g_+-\wtilde g_+\|_1 + \|g_--\wtilde g_-\|_1
   \le \epsilon.
\]
This finishes the proof for the case $\|f_0\|_\infty\le 1$.
Otherwise we repeat steps 2 and 4 where $f_0$ is replaced with
$\wtilde f_0:=f_0/c_0$ for $c_0:=\|f_0\|_\infty<\infty$,
and where bound $\epsilon/5$ is reduced to $\epsilon/(5c_0)$.
In particular, $\wtilde f_0=f_+-f_-$ and $\wtilde g:=c_0(\wtilde g_+-\wtilde g_-)$ satisfies
$\wtilde g\in C^\infty_0(\Omega\setminus L)$ and
\[
   \|f-\wtilde g\|_1
   \le
   \|f-f_0\|_1
   + c_0
   \bigl(\|f_+ -g_+\|_1 + \|f_- - g_-\|_1 + \|g_+-\wtilde g_+\|_1 + \|g_--\wtilde g_-\|_1\Bigr)
   \le \epsilon.
\]
\end{proof}

\begin{prop}[inclusion] \label{prop_cont}
The inclusion
\begin{align*} 
   \wtilde H^{1/2}_{0,n}(\cS;\Rt)\oplus\wtilde H^{1/2}_{0,t}(\cS;\Rt) &\subset H^{1/2}_0(\cS;\Rt)
\end{align*}
is dense. It is strict if there is a face $F\in\cF(\Omega)$ that does not touch
the boundary. Moreover, every $\tau\in H(\div,\mesh;\SS)$ satisfies
\[
   \tau\in H(\div,\Omega;\SS)
   \quad\Leftrightarrow\quad
   \dual{\rho}{\tau}_\cS = 0\quad
   \forall \rho\in\wtilde H^{1/2}_{0,n}(\cS;\Rt)\oplus\wtilde H^{1/2}_{0,t}(\cS;\Rt).
\]
\end{prop}

\begin{proof}
The proof is analogous to the proof of \cite[Proposition~2, (3)]{CarstensenH_NNC}
by replacing density relation \cite[Lemma~17.3]{Tartar_07_ISS} with Proposition~\ref{prop_Tartar}.
\end{proof}

\subsection{Normal-normal continuous stress space and dual trace operator} \label{sec_trace_nn}

Duality \eqref{duality} and normal trace space \eqref{H_trace_nor} provide us
with the definition of a space of pointwise symmetric stresses with continuous normal-normal
components across $\cS$,
\[
   H_{nn}(\div,\mesh;\SS) :=
   \{\tau\in H(\div,\mesh;\SS);\;
     \dual{\varphi}{\tau}_\cS=0\ \forall \varphi\in \wtilde H^{1/2}_{0,n}(\cS;\Rt)\}.
\]
It is a closed subspace of $H(\div,\mesh;\SS)$.
The duality with this constrained space reduces the canonical trace operator $\gamma$ to
\begin{align*}
   \trnn:\;&\left\{\begin{array}{cll}
               H^1(\Omega;\Rt) &\rightarrow& H_{nn}(\div,\mesh;\SS)^*,\\
               v &\mapsto& \dual{\trnn(v)}{\tau}_\cS
               := \vdual{\strain{v}}{\tau} + \vdual{v}{\div\tau}_\mesh.
            \end{array}\right.
\end{align*}
It delivers tangential traces on interior faces and canonical traces on faces $F\subset\Gamma$.
The resulting trace spaces (the latter one including homogeneous Dirichlet data) are
\begin{align*}
   H^{1/2}_t(\cS;\Rt) &:= \trnn\bigl(H^1(\Omega;\Rt)\bigr)\quad\text{and}\quad
   H^{1/2}_{0,t}(\cS;\Rt) := \trnn\bigl(H^1_0(\Omega;\Rt)\bigr).
\end{align*}
It remains to specify norms.
For an element $\tau$ of the product space $H(\div,\mesh;\SS)$ we use the product norm
$\|\tau\|_{\div,\mesh}:=\Bigl(\sum_{\el\in\mesh}\|\tau\|_{\div,\el}^2\Bigr)^{1/2}$.
Trace space $H^{1/2}_{t}(\cS;\Rt)$ is endowed with the canonical trace norm
\[
   \|\rho\|_{1/2,t,\cS} :=
   \mathrm{inf}
   \{\|v\|_1;\; v\in H^1(\Omega;\Rt),\ \trnn(v)=\rho\},
   \quad\rho\in H^{1/2}_{t}(\cS;\Rt).
\]
The following conformity and norm relations are essential to develop and analyze
our mixed formulation of \eqref{model}.

\begin{prop}[duality] \label{prop_conf}
(i) Any $\tau\in H_{nn}(\div,\mesh;\SS)$ satisfies
\[
   \tau\in H(\div,\Omega;\SS)
   \quad\Leftrightarrow\quad
   \dual{\rho}{\tau}_\cS = 0\quad
   \forall \rho\in H^{1/2}_{0,t}(\cS;\Rt).
\]
(ii) Any $\rho\in H^{1/2}_{t}(\cS;\Rt)$ satisfies
\[
   \|\rho\|_{1/2,t,\cS} =
   \sup_{0\not=\tau\in H_{nn}(\div,\mesh;\SS),\; \vdual{\tr(\tau)}{1}=0}
   \frac {\dual{\rho}{\tau}_\cS}{\|\tau\|_{\div,\mesh}}.
\]
\end{prop}

\begin{proof}
In two dimensions, statement (i) is \cite[Lemma~5]{CarstensenH_NNC}.
Its proof is based on a decomposition of (a dense subspace of)
$H^{1/2}(\cS;\R^2)$ into normal and tangential trace spaces. Proposition~\ref{prop_cont}
does provide such a decomposition in three dimensions and allows to apply the same
technique as in \cite{CarstensenH_NNC}.
For an abstract setting of this result see \cite[Lemma~27]{Heuer_GMP}.
Also statement (ii) is proved analogously to the two-dimensional
case \cite[Proposition~6]{CarstensenH_NNC}, and for an abstract form
(without the trace constraint for test functions), see \cite[Lemma~26]{Heuer_GMP}.
\end{proof}

\section{Mixed formulation} \label{sec_mixed}

An application of trace operator $\gamma_{1,nn}$ and the conformity characterization
given by Proposition~\ref{prop_conf}(i) lead to the following mixed formulation
of \eqref{model}.
\emph{Find $\sigma\in H_{nn}(\div,\mesh;\SS)$, $u\in L_2(\Omega;\Rt)$, and
$\eta\in H^{1/2}_{0,t}(\cS;\Rt)$ such that}
\begin{subequations} \label{VF}
\begin{alignat}{3}
   &\vdual{\AA\sigma}{\dsigma} + \vdual{u}{\div\dsigma}_\mesh - \dual{\eta}{\dsigma}_\cS
   &&= 0, \label{VFa}\\
   &\vdual{\div\sigma}{\du}_\mesh - \dual{\deta}{\sigma}_\cS
   &&= -\vdual{f}{\du} \label{VFb}
\end{alignat}
\end{subequations}
\emph{for any $\dsigma\in H_{nn}(\div,\mesh;\SS)$, $\du\in L_2(\Omega;\Rt)$, and
$\deta\in H^{1/2}_{0,t}(\cS;\Rt)$.}

We show its well-posedness.

\begin{theorem}[well-posedness of mixed formulation] \label{thm_VF}
Given $f\in L_2(\Omega;\Rt)$, problem \eqref{VF} is well posed
with a unique solution $(\sigma,u,\eta)$.
The couple $(u,\sigma)\in H^1(\Omega)\times H(\div,\Omega;\SS)$ solves \eqref{model},
$\eta=\trnn(u)$, and
\begin{align*}
   \|\sigma\|_{\div,\mesh}
   + \|u\| + \|\eta\|_{1/2,t,\cS} &\lesssim \|f\|
\end{align*}
holds uniformly with respect to $\lambda\in (0,\infty)$ and $\mesh$.
\end{theorem}

\begin{proof}
The proof uses the standard ingredients of mixed formulations, and is
analogous to the two-dimensional case studied in \cite[Theorem~10]{CarstensenH_NNC}.
We recall the main steps and give some more details on the uniform stability.

We need the representation
\begin{align} \label{coercive}
   \vdual{\AA\tau}{\dtau}
   =
   \frac 1{2\mu} \vdual{\dev\tau}{\dev\dtau}
   +
   \frac 1{3(3\lambda+2\mu)} \vdual{\tr(\tau)}{\tr(\dtau)}
   \quad\forall \tau,\dtau\in L_2(\Omega;\SS),
\end{align}
see, e.g., \cite[(9.1.8)]{BoffiBF_13_MFE}.
The selection of $\dsigma=\id$ in \eqref{VFa} and relation
$\dual{\eta}{\id}_\cS=\dual{\eta}{n}_\Gamma=0$ for $\eta\in H^{1/2}_{0,t}(\cS;\Rt)$
show that $\vdual{\tr(\sigma)}{1}=0$. We conclude that formulation \eqref{VF}
is equivalent to the same system where $H_{nn}(\div,\mesh;\SS)$ is replaced with
\[
   \wtilde H_{nn}(\div,\mesh;\SS) := \{
   \tau\in H_{nn}(\div,\mesh;\SS);\; \vdual{\tr(\tau)}{1}=0\}.
\]
We continue to analyze this formulation.

The bilinear and linear forms of \eqref{VF} are uniformly bounded by the selection
of norms. Inf-sup properties
\begin{align*} 
   \sup_{0\not=\tau\in H(\div,\Omega;\SS)}
   \frac {\vdual{\div\tau}{\du}}{\|\tau\|_{\div}}
   \gtrsim
   \|\du\|\quad\forall\du\in L_2(\Omega;\Rt)
\end{align*}
and
\begin{align*} 
   \sup_{0\not=\tau\in \wtilde H_{nn}(\div,\mesh;\SS)}
   \frac {\dual{\deta}{\tau}_\cS}{\|\tau\|_{\div,\mesh}}
   \ge
   \|\deta\|_{1/2,t,\cS}\quad\forall\deta\in H^{1/2}_{0,t}(\cS;\Rt)
\end{align*}
hold by surjectivity of $\div:\;H(\div,\Omega;\SS)\to L_2(\Omega;\Rt)$
and Proposition~\ref{prop_conf}(ii), respectively.
The conformity relation of Proposition~\ref{prop_conf}(i)
and \cite[Theorem~3.3]{CarstensenDG_16_BSF} then show the required combined inf-sup relation
\begin{align*}  
   \sup_{0\not=\tau\in \wtilde H_{nn}(\div,\mesh;\SS)}
   \frac {\vdual{\div\tau}{\du}_\mesh -\dual{\deta}{\tau}_\cS}{\|\tau\|_{\div,\mesh}}
   \gtrsim
   \|\du\| + \|\deta\|_{1/2,t,\cS}\quad
\end{align*}
for any $\du\in L_2(\Omega;\Rt)$ and $\deta\in H^{1/2}_{0,t}(\cS;\Rt)$.

For the uniform well-posedness of \eqref{VF} it remains to verify the
uniform coercivity of bilinear form $\vdual{\AA\cdot}{\cdot}$ on the kernel
\begin{align*}
   K_0 :=
   &\{\tau\in \wtilde H_{nn}(\div,\mesh;\SS);\;
     \vdual{\div\tau}{\du}_\mesh=0\ \forall\du\in L_2(\Omega;\Rt),\;
     \dual{\deta}{\tau}_\cS=0\ \forall\deta\in H^{1/2}_{0,t}(\cS;\Rt)\}\\
   = &\{\tau\in H(\div,\Omega;\SS);\; \div\tau=0,\ \vdual{\tr(\tau)}{1}=0\}.
\end{align*}
The latter identity is due to Proposition~\ref{prop_conf}(i).
Critical is the bound
\[
   \|\tr(\tau)\| \le
   C \bigl( \|\dev\tau\| + \|\div\tau\|_{-1}\bigr)
   = C \|\dev\tau\|\quad\forall \tau\in K_0
\]
for a constant $C>0$, see \cite[Theorem~1]{CarstensenH_FOT}.
The uniform coercivity of $\vdual{\AA\cdot}{\cdot}$ on $K_0$
then follows from \eqref{coercive}.
The fact that $u\in H^1(\Omega;\Rt)$ and $\sigma\in H(\div,\Omega;\SS)$ solve
\eqref{model} is immediate, and
$\eta=\trnn(u)$ holds by \eqref{VFa} and the definition of $\trnn$.
\end{proof}
\section{Finite element approximation} \label{sec_fem}

Let us introduce some polynomial related notation.
For $K\in\cF\cup\mesh$ and $U\in\{\R,\Rt,\SS\}$ we denote by $P^k(K;U)$
the space of polynomials of degree $k$ on $K$ with values in $U$.
The piecewise polynomial spaces are $P^k(\mesh;U)$. We drop $U$ when $U=\R$.

\subsection{An $H(\div;\SS)$ element} \label{sec_el}

Let $T\subset\Rt$ be a tetrahedron with vertices $x_i$, barycentric coordinates $\lambda_i$,
faces $F_i$ opposite to $x_i$, and unit exterior normal vectors $n_i$ on $F_i$
(we also employ the generic notation $n$),
and denote $\wtilde n_i:=\grad\lambda_i\in\mathrm{span}\{n_i\}$, $i=1,\ldots,4$.
We also need unit vectors $e_{i,j}$ (in a certain direction) generated by the edges
$\overline{F}_i\cap\overline{F}_j$, $i\not=j=1,\ldots,4$.
All these simplex-related objects are numbered modulus $4$, e.g., $x_5=x_1$.
We abbreviate $a\odot b:= (a b^\top+b a^\top)/2\in\SS$ for vectors $a,b\in\Rt$.

The construction of our $H(\div,\el;\SS)$-element is based upon the constant tensors
\[
   T_i:=\frac{e_{i+1,i+2}\odot e_{i+1,i+3}}{n_i\cdot e_{i+1,i+2}\; n_i\cdot e_{i+1,i+3}}
   \quad (i=1,\ldots,4),\quad
   T_5:= e_{1,2}\odot e_{3,4},\quad
   T_6:= e_{1,3}\odot e_{2,4}
\]
which from a basis of $\SS$ (as seen below in Lemma~\ref{la_SS}),
and the polynomial tensor spaces
\begin{align*}
   P^0_{nn}(\el;\SS) &:= \mathrm{span}\bigl\{T_1,\ldots,T_4\bigr\},\quad
   P^0_0(\el;\SS) := \mathrm{span}\bigl\{T_5,T_6\bigr\},\\
  P^{2,k}(\el,\SS)
  &:=\{\tau\in P^2(\el;\SS);\; n\cdot\tau n|_F\in P^k(F),\ F\in\cF(\el)\},
  \quad k\in\{0,1\},\\
  P^2_0(\el;\SS)&:=\{\tau\in P^2(\el;\SS);\; n\cdot\tau n|_{\partial\el}=0\}.
\end{align*}
{\bf Remark.}
\emph{The construction of $T_i\in\SS$ is such that
$T_i n_{i+1}=0$, $i=1,\ldots,4$, that is, the image of mapping
$T_i:\;\Rt\to\Rt$ is the plane through the origin that is parallel to face $F_{i+1}$.
Furthermore, $n_j\cdot T_i n_j=0$ for $j=1,\ldots,4$, $i=5,6$, and
we will see that $n_j\cdot T_i n_j=\delta_{ij}$ for $i,j=1,\ldots,4$.}

Our $H(\div,\el;\SS)$ finite element on the tetrahedron $\el$ is the vector space
\begin{equation} \label{element}
   \XX{\el} :=
   \bigl(P^2(\el) P^0_{nn}(\el;\SS)\cap P^{2,1}(\el,\SS)\bigr)\oplus P^0_0(\el;\SS)
\end{equation}
and the following degrees of freedom,
\begin{subequations} \label{dof_el}
\begin{alignat}{3}
   \label{dof_el_nn}
   &|F| \dual{n\cdot\tau n}{q}_{F}\quad &&(q\in P^1(F),\ F\in\cF(\el)),\\
   \label{dof_el_id}
   &\vdual{\tau}{c}_\el        &&(c\in P^0(\el;\SS)),\\
   \label{dof_el_div}
   &\vdual{\div\tau}{v}_\el    &&(v\in P^1(\el;\Rt)).
\end{alignat}
\end{subequations}
The proof that ($\el$,$\XX{\el}$,\eqref{dof_el}) is a finite element requires some preparation.

\begin{lemma} \label{la_SS} We have
\(
   P^0(\el;\SS) = P^0_{nn}(\el;\SS) \oplus P^0_0(\el;\SS).
\)
\end{lemma}

\begin{proof}
We start by proving that $\{T_1,\ldots,T_6\}$ is a basis of $P^0(\el;\SS)$.
Since $\dim P^0(\el;\SS)=6$ and $T_i\in P^0(\el;\SS)$, $i=1,\ldots,6$, it suffices
to show that the tensors $T_i$ are linearly independent.

For every $i\in\{1,\ldots,4\}$,
at least one of the two vectors $e_{i+1,i+2}$ and $e_{i+1,i+3}$ is tangential to the faces
$F_j$ for $j\in\{i+1,i+2,i+3\}$. Furthermore, both vectors are not tangential to $F_i$.
Consequently
\[
   (n_i\cdot e_{i+1,i+2})(n_i\cdot e_{i+2,i+3})\not=0 \quad\text{and}\quad
   (n_j\cdot e_{i+1,i+2})(n_j\cdot e_{i+1,i+3})=0\quad \text{for}\ j\not=i.
\]
Therefore, tensors $T_i$ are well defined and satisfy
\[
   n_j\cdot T_i n_j=\delta_{ij}\quad\text{for}\ i,j=1,\ldots,4.
\]
On the other hand, for each face $F_i$, one of the vectors $e_{1,2},e_{3,4}$
and one of the vectors $e_{1,3},e_{2,4}$ is tangential to $F_i$.
This implies that
\[
   n_j\cdot T_i n_j=0\quad\text{for}\ i=5,6,\ j=1,\ldots,4.
\]
It follows that $T_1,\ldots,T_6$ are linearly independent, and thus form a basis
of $P^0(\el;\SS)$, if $T_5$ and $T_6$ are linearly independent.
The latter is true because
\[
   n_3\cdot T_5 n_2 = n_3\cdot e_{1,2}\; e_{3,4}\cdot n_2\not=0
   \quad\text{but}\quad
   n_3\cdot T_6 n_2 = n_3\cdot e_{1,3}\; e_{2,4}\cdot n_2=0.
\]
The splitting of $P^0(\el;\SS)$ then holds by definition of the two subspaces.
\end{proof}

We use a transformation of $\el\in\mesh$ to a reference tetrahedron $\elref$
to show that \eqref{dof_el} are degrees of freedom of $\XX{\el}$.
We consider the tetrahedron $\elref$ with vertices
$\xref_1=(0,0,0)$, $\xref_2=(1,0,0)$, $\xref_3=(0,1,0)$, $\xref_4=(0,0,1)$,
and let $G_\el:\;\elref\to\el$ denote the affine diffeomorphism
\[
   \begin{pmatrix}\xref\\ \yref\\ \zref\end{pmatrix} \mapsto
   a_\el + B_\el \begin{pmatrix}\xref\\ \yref\\ \zref\end{pmatrix}
\]
with $a_\el\in\Rt$, $B_\el\in\Rtt$, and set $J_\el:=\det(B_\el)$.
We use the generic notation $\nref$ for exterior normal vectors
on faces $\Fref\in\cF(\elref)$.
Tensors $\tauref,\,\sigmaref:\;\elref\to\SS$ and vector functions $\vref:\;\elref\to\Rt$
are transformed onto elements $\el\in\mesh$ as
$\PiolaK{\el}(\tauref)=\tau$,
$\PiolaKc{\el}(\sigmaref)=\sigma$ and $\PiolaCurl{\el}(\vref)=v$, where
\[
   J_\el^2\tau\circ G_\el := B_\el\tauref B_\el^\top,\quad
   \sigma\circ G_\el := J_\el B_\el^{-\top}\sigmaref B_\el^{-1},
   \quad\text{and}\quad
   v\circ G_\el:=J_\el B_\el^{-\top}\vref
   \quad\text{on}\ \elref.
\]
These are re-scaled Piola--Kirchhoff and Piola transformations.
They conserve the properties of space $\XX{\el}$ and degrees of freedom
\eqref{dof_el}, as seen next.

\begin{lemma} \label{la_trafo}
Let $\el\in\mesh$ and $\tau\in L_2(\el;\SS)$ be given with $\tau=\PiolaK{\el}(\tauref)$.

(i) Property $\tau\in\XX{\el}$ holds if and only if $\tauref\in\XX{\elref}$.

(ii) For $F\in\cF(\el)$, $q\in P^2(F)$, $c\in P^0(\el;\SS)$, and $v\in P^1(\el;\Rt)$ let
     $\Fref\in\cF(\elref)$, $\widehat q\in P^2(\widehat F)$, $\widehat c\in P^0(\elref;\SS)$, and
     $\widehat v\in P^1(\elref;\Rt)$ be the transformed objects with
     $F=G_\el(\Fref)$, $q\circ G_\el=\widehat q$, $c=\PiolaKc{\el}(\widehat c)$,
     and $v=\PiolaCurl{\el}(\vref)$. Then,
\begin{align*}
   &|F|\dual{n\cdot\tau n}{q}_{F}
   = |\Fref|\dual{\nref\cdot\tauref\nref}{\widehat q}_{\Fref},\quad
   \vdual{\tau}{c}_{\el}=\vdual{\tauref}{\widehat c}_\elref,\quad
   \vdual{\div\tau}{v}_{\el}=\vdual{\divref\tauref}{\vref}_\elref.
\end{align*}
\end{lemma}

\begin{proof}
A proof of statement (ii) is analogous to that of the two-dimensional case on
triangles, see~\cite[Lemma~13]{CarstensenH_NNC}. Statement (i) is consequence of
relations (ii). In fact, the Piola--Kirchhoff transformation $\PiolaK{\el}$
and its inverse conserve the space of symmetric polynomial tensor functions of a fixed degree.
The conservation of the moments of the normal-normal traces implies that
$\PiolaK{\el}$ is a bijective mapping between $P^0_{nn}(\elref;\SS)$  and $P^0_{nn}(\el;\SS)$,
and between $P^0_0(\elref;\SS)$ and $P^0_0(\el;\SS)$.
It also means that the normal-normal traces of $\tau$ on faces $F\in\cF(\el)$
are linear polynomials if and only if this applies to the normal-normal traces of $\tauref$ on
$\Fref\in\cF(\elref)$. This finishes the proof.
\end{proof}

\begin{prop}[degrees of freedom] \label{prop_dof}
For a tetrahedron $\el\subset\Rt$ relations
$P^0(\el;\SS)\subset \XX{\el}$, $\dim\XX{\el}=30$ hold, and
\eqref{dof_el} are degrees of freedom of $\XX{\el}$ for a finite element in the sense of Ciarlet.
\end{prop}

\begin{proof}
The first relation follows from of Lemma~\ref{la_SS} and
the fact that $P^0_{nn}(\el;\SS)$ and $P^0_0(\el;\SS)$ are subsets of $\XX{\el}$ by construction.

{\bf Dimension.}
By definition of $P^0_{nn}(\el;\SS)$,
\[
   P^2(\el) P^0_{nn}(\el;\SS) = T_1P^2(\el)\oplus \cdots \oplus T_4P^2(\el)
\]
has dimension $40$.
Since $n_j\cdot(T_iP^2(\el)) n_j=\delta_{ij}P^2(\el)$ for $i,j=1,\ldots,4$,
the constraint
$\bigl(T_iP^2(\el)\bigr)\cap P^{2,1}(\el,\SS)$ requires to select the subspaces
\begin{align*}
   P^{2,1}_i(\el)
   &:=\{v\in P^2(\el);\; v|_{F_i}\in P^1(F_i)\} 
   = P^1(\el)\oplus
   \mathrm{span}\{\lambda_i\lambda_{i+1},\lambda_i\lambda_{i+2},\lambda_i\lambda_{i+3}\}
\end{align*}
for $i=1,\ldots,4$, each of dimension $7$. Summing the dimensions gives $4$ times $7$
plus the dimension $2$ of $P^0_0(\el;\SS)$ in \eqref{element}.

{\bf Degrees of freedom.}
By Lemma~\ref{la_trafo}, \eqref{dof_el} are degrees of freedom of $\XX{\el}$
for a general tetrahedron $\el\in\mesh$ if and only if they are degrees of freedom
for the reference tetrahedron $\el=\elref$.
In the remainder of this proof we therefore assume without loss of generality
that $\el=\elref$. We need this assumption only in the final step
to verify the invertibility of a $12\times 12$-matrix.
The number of functionals in \eqref{dof_el} is
$30$: 12 in \eqref{dof_el_nn}, 6 in \eqref{dof_el_id}, and 12 in \eqref{dof_el_div}.
It is therefore enough to show the linear independence of \eqref{dof_el} on $\XX{\el}$.

{\bf Linear independence.}
Let $\tau\in\XX{\el}$ be given with vanishing functionals \eqref{dof_el}.
We have to show that $\tau=0$.
The vanishing of \eqref{dof_el_nn} leads to
\[
   \tau\in \oplus_{i=1}^4 \bigr(T_iP^2(\el)\cap P^2_0(\el,\SS)\bigr)\oplus P^0_0(\el;\SS).
\]
Since
\begin{align*}
   &T_iP^2(\el)\cap P^2_0(\el,\SS)
   = T_i\,\mathrm{span}\{\phi\in P^2(\el);\; \phi|_{F_i}=0\}
   = T_i\,\mathrm{span}\{\lambda_i, \lambda_i\lambda_{i+1}, \lambda_i\lambda_{i+2},
                       \lambda_i\lambda_{i+3}\}
\end{align*}
for $i=1,\ldots,4$ we deduce the representation
\begin{align*}
   \tau=\sum_{i=1}^4 \lambda_i\phi_i T_i+\phi_5T_5+\phi_6T_6\quad\text{with}\quad
   \phi_i\in P^1(\el) \text{ for } i=1,\ldots,4,\ \text{ and } \phi_5,\phi_6\in\R.
\end{align*}
The vanishing of degrees of freedom \eqref{dof_el_id} means
\[
   \sum_{i=1}^4 \vdual{\lambda_i\phi_i}{1}_\el T_i
             +\vdual{\phi_5}{1}_\el T_5+\vdual{\phi_6}{1}_\el T_6=0,
\]
so that, by the linear independence of $T_1,\ldots,T_6$,
\begin{align*}
  \vdual{\lambda_i\phi_i}{1}_\el=0\quad \text{for } i=1,\ldots,4,\quad
  \text{and}\quad \phi_5=\phi_6=0.
\end{align*}
We conclude that
\begin{align} \label{tau1}
   \tau=\sum_{i=1}^4 \lambda_i\phi_i T_i
   \quad\text{and}\quad \vdual{\lambda_j\phi_j}{1}_\el=0\quad \text{for}\ j=1,\ldots,4.
\end{align}
In the following we represent $\phi_i=\sum_{j=1}^4 \phi_i(x_j) \lambda_j$ and
calculate
\(
   \grad\phi_i=\sum_{j=1}^4 \phi_i(x_j)\wtilde n_j
\)
for $i=1,\ldots,4$ (recall that $\wtilde n_j=\grad\lambda_j\in\mathrm{span}(n_j)$).
Representation \eqref{tau1} and the vanishing of \eqref{dof_el_div} reveal that
\begin{align} \label{sys}
   0=\div\tau(x_k) &= \sum_{i=1}^4 T_i\bigl(\phi_i\grad\lambda_i + \lambda_i\grad\phi_i\bigr)(x_k)
   =
   \sum_{i=1}^4 T_i\bigl(\phi_i\wtilde n_i
                       + \lambda_i\sum_{j=1}^4 \phi_i(x_j)\wtilde n_j\bigr)(x_k)\nonumber\\
   &=
   \sum_{i=1}^4 T_i\bigl(\phi_i(x_k)\wtilde n_i
                       + \delta_{ik}\sum_{j=1}^4 \phi_i(x_j)\wtilde n_j\bigr)
   \quad\text{for}\ k=1,\ldots,4.
\end{align}
These equations form a homogeneous linear system of $12$ equations for the $16$ unknowns
$\phi_i(x_j)$, $i,j=1,\ldots,4$. Instead of unknowns $\phi_i(x_i)$ we use the scaled unknowns
$2\phi_i(x_i)$ for $i=1,\ldots,4$ and order the equations with respect to $k$ and
the unknowns with inner loop over $j$ and outer loop over $i$, namely
\[
   2\phi_1(x_1), \phi_1(x_2), \phi_1(x_3), \phi_1(x_4),\quad
   \phi_2(x_1), 2\phi_2(x_2), \phi_2(x_3), \phi_2(x_4),\quad\ldots,
   \phi_4(x_3), 2\phi_4(x_4).
\]
Denoting $t_{ij}:=T_i\wtilde n_j$, the corresponding matrix of system \eqref{sys} becomes
\[
   \begin{pmatrix}
   \begin{array}{cccc|cccc|cccc|cccc}
      t_{11} & t_{12} & t_{13} & t_{14} & t_{22} & 0 & 0 & 0 & t_{33} & 0 & 0 & 0 & t_{44} & 0 & 0 & 0 \\\hline
      0 & t_{11} & 0 & 0 & t_{21} & t_{22} & t_{23} & t_{24} & 0 & t_{33} & 0 & 0 & 0 & t_{44} & 0 & 0 \\\hline
      0 & 0 & t_{11} & 0 & 0 & 0 & t_{22} & 0 & t_{31} & t_{32} & t_{33} & t_{34} & 0 & 0 & t_{44} & 0 \\\hline
      0 & 0 & 0 & t_{11} & 0 & 0 & 0 & t_{22} & 0 & 0 & 0 & t_{33} & t_{41} & t_{42} & t_{43} & t_{44}
   \end{array}
   \end{pmatrix}.
\]
It remains to incorporate the relations
$\vdual{\lambda_i\phi_i}{1}_\el=0$ for $i=1,\ldots,4$ from \eqref{tau1}.
We note that
$\vdual{\lambda_i}{\lambda_j}_\el/\vdual{\lambda_i}{\lambda_{i+1}}_\el=\delta_{ij}+1$,
see, e.g., \cite[Theorem~3.1]{VermolenS_18_IRP},
and recall the representation $\phi_i=\sum_{j=1}^4 \phi_i(x_j)\lambda_j$.
We find that
\(
   \vdual{\lambda_i\phi_i}{1}_\el = \sum_{j=1}^4 \phi_i(x_j) \vdual{\lambda_i}{\lambda_j}_\el=0
\)
holds if and only if
\[
   \frac 1{\vdual{\lambda_i}{\lambda_{i+1}}_\el}
   \sum_{j=1}^4 \phi_i(x_j) \vdual{\lambda_i}{\lambda_j}_\el
   =
   2\phi_i(x_i) + \phi_i(x_{i+1}) + \phi_i(x_{i+2}) + \phi_i(x_{i+3}) = 0,
   \quad i=1,\ldots,4.
\]
We use this relation to eliminate the unknowns $2\phi_i(x_i)$ for $i=1,\ldots,4$.
With the notation
$\wtilde t_{ij}:=t_{ij}-t_{ii}$ the reduced $12\times 12$ matrix reads
\def\mm{\!\!-\!\!}
\newcommand\Tstrut{\rule{0pt}{2.6ex}} 
\begin{equation} \label{mat_red}
\begin{pmatrix}
\begin{array}{ccc|ccc|ccc|ccc}
  \wtt_{12}& \wtt_{13}& \wtt_{14}& t_{22} & 0 & 0 & t_{33} & 0 & 0 & t_{44} & 0 & 0 \\\hline
\Tstrut
  t_{11} & 0 & 0 & \wtt_{21}& \wtt_{23}& \wtt_{24}& 0 & t_{33} & 0 & 0 & t_{44} & 0 \\\hline
\Tstrut
  0 & t_{11} & 0 & 0 & t_{22} & 0 & \wtt_{31}& \wtt_{32}& \wtt_{34}& 0 & 0 & t_{44} \\\hline
\Tstrut
  0 & 0 & t_{11} & 0 & 0 & t_{22} & 0 & 0 & t_{33} & \wtt_{41}& \wtt_{42}& \wtt_{43}
\end{array}
\end{pmatrix}.
\end{equation}
The proof of the lemma is finished by noting that this matrix is invertible.
Let us give the details.

Recall that we are considering the reference tetrahedron $\el=\elref$
with the previously introduced numbering of vertices.
In particular,
\begin{align*}
   &\lambda_1(x,y,z)=1-x-y-z,\ \lambda_2(x,y,z)=x,\ \lambda_3(x,y,z)=y,\ \lambda_4=z,\\
   &\wtilde n_1=-(1,1,1)^\top,\ \wtilde n_2=(1,0,0)^\top,\ \wtilde n_3=(0,1,0)^\top,\ \wtilde n_4=(0,0,1)^\top,
\end{align*}
and selecting the edge vectors
\begin{align*}
   &e_{1,2}=(0,1,-1)^\top,\ e_{1,3}=(1,0,-1)^\top,\ e_{1,4}=(1,-1,0)^\top,\ e_{2,3}=(0,0,1)^\top,\\
   &e_{2,4}=(0,1,0)^\top,\ e_{3,4}=(1,0,0)^\top,\ e_{i,j}=e_{j,i}\ (i>j)
\end{align*} 
we calculate
\begin{align*}
   &T_1=\frac{e_{2,3}\odot e_{2,4}}{n_1\cdot e_{2,3}\; n_1\cdot e_{2,4}}
   = \begin{pmatrix} 0 & 0 & 0\\ 0 & 0 & 3/2\\ 0 & 3/2 & 0 \end{pmatrix},
   &&T_2=\frac{e_{3,4}\odot e_{3,1}}{n_2\cdot e_{3,4}\; n_2\cdot e_{3,4}}
   = \begin{pmatrix} 1 & 0 & -1/2\\ 0 & 0 & 0\\ -1/2 & 0 & 0 \end{pmatrix},\\
   &T_3=\frac{e_{4,1}\odot e_{4,2}}{n_3\cdot e_{4,1}\; n_3\cdot e_{4,2}}
   = \begin{pmatrix} 0 & -1/2 & 0\\ -1/2 & 1 & 0\\ 0 & 0 & 0 \end{pmatrix},
   &&T_4=\frac{e_{1,2}\odot e_{1,3}}{n_4\cdot e_{1,2}\; n_4\cdot e_{1,3}}
   = \begin{pmatrix} 0 & 1/2 & -1/2\\ 1/2 & 0 & -1/2\\ -1/2 & -1/2 & 1 \end{pmatrix}.
\end{align*}
This gives (note that $t_{i,i+1}=(0,0,0)^\top$)
\newcommand{\mvec}[3]{\begin{pmatrix}#1\\#2\\#3\end{pmatrix}}
\begin{align*}
   &t_{11}=\mvec{0}{-3/2}{-3/2},\ t_{13}=\mvec{0}{0}{3/2},\ t_{14}=\mvec{0}{3/2}{0},\
   &&\wtt_{12}=-t_{11},\ \wtt_{13}=\mvec{0}{3/2}{3},\ \wtt_{14}=\mvec{0}{3}{3/2},\\
   &t_{22}=\mvec{1}{0}{-1/2},\ t_{21}=\mvec{-1/2}{0}{1/2},\ t_{24}=\mvec{-1/2}{0}{0},\
   &&\wtt_{21}=\mvec{-3/2}{0}{1},\ \wtt_{23}=-t_{22},\ \wtt_{24}=\mvec{-3/2}{0}{1/2},\\
   &t_{33}=\mvec{-1/2}{1}{0},\ t_{31}=\mvec{1/2}{-1/2}{0},\ t_{32}=\mvec{0}{-1/2}{0},\
   &&\wtt_{31}=\mvec{1}{-3/2}{0},\ \wtt_{32}=\mvec{1/2}{-3/2}{0},\ \wtt_{34}=-t_{33},\\
   &t_{44}=\mvec{-1/2}{-1/2}{1},\ t_{42}=\mvec{0}{1/2}{-1/2},\ t_{43}=\mvec{1/2}{0}{-1/2},\
   &&\wtt_{41}=-t_{44},\ \wtt_{42}=\mvec{1/2}{1}{-3/2},\ \wtt_{43}=\mvec{1}{1/2}{-3/2},
\end{align*}
and matrix \eqref{mat_red} reads
\[
\begin{pmatrix}
\begin{array}{ccc|ccc|ccc|ccc}
%
%
%
  0  & 0  & 0  &     1   & 0  & 0    &     -1/2 & 0 & 0        & -1/2 & 0 & 0 \\
  3/2& 3/2& 3  &     0   & 0  & 0    &      1   & 0 & 0        & -1/2 & 0 & 0 \\
  3/2& 3  & 3/2&     -1/2& 0  & 0    &      0   & 0 & 0        & 1    & 0 & 0 \\\hline

  0 & 0 & 0    &     -3/2& -1 & -3/2 &      0   & -1/2 & 0     & 0 & -1/2 & 0 \\
  -3/2 & 0& 0  &     0 &  0 & 0      &      0   & 1 & 0        & 0 & -1/2 & 0 \\
  -3/2 & 0& 0  &     1 & 1/2& 1/2    &      0   & 0 & 0        & 0 & 1 & 0    \\\hline

  0 & 0  & 0   &     0 & 1 & 0       &       1  & 1/2 & 1/2    & 0 & 0 & -1/2 \\
  0 & -3/2 & 0 &     0 & 0 & 0       &     -3/2 & -3/2 & -1    & 0 & 0 & -1/2 \\
  0 & -3/2 & 0 &     0 & -1/2 & 0    &       0  & 0 & 0        & 0 & 0 & 1    \\\hline

  0 & 0 & 0    &     0 & 0 & 1       &       0  & 0 & -1/2     & 1/2& 1/2& 1  \\
  0 & 0 & -3/2 &     0 & 0 & 0       &       0  & 0 & 1        & 1/2& 1& 1/2  \\
  0 & 0 & -3/2 &     0 & 0 & -1/2    &       0  & 0 & 0        & -1& -3/2& -3/2
\end{array}
\end{pmatrix}.
\]
A lengthy calculation shows that its determinant is $2^{-8} 3^4 5^2$.
This finishes the proof.
\end{proof}

\subsection{Approximation space} \label{sec_approx}

Elements $\XX{\el}$ ($\el\in\mesh$) from \eqref{element} generate the product
space $\XX{\mesh}$ that gives rise to
the $H_{nn}(\div,\mesh;\SS)$-conforming finite element space
\[
   P^2_{nn}(\mesh) := \XX{\mesh}\cap H_{nn}(\div,\mesh;\SS)
   = \{\tau\in\XX{\mesh};\; [n\cdot\tau n]|_F=0\ \forall F\in\cF(\Omega)\}.
\]
Here, $[n\cdot\tau n]|_F$ denotes the normal-normal jump of $\tau$ across $F$.
Equivalently, this conformity is achieved by the selection of unique degrees
of freedom \eqref{dof_el_nn} assigned to interior faces $F\in\cF(\Omega)$.
By also assigning volume-related degrees of freedom \eqref{dof_el_id}, \eqref{dof_el_div}
we obtain an interpolation operator
\[
   \Inn{\mesh}:\; H(\div,\Omega;\SS)\cap L_q(\Omega;\SS) \to P^2_{nn}(\mesh)
   \quad (q>2).
\]
In the following, $\div_\mesh$ is the piecewise divergence operator and
$\Pi^1_\mesh$ denotes the $L_2$-projection onto piecewise linear vector-valued polynomials.

\begin{prop}[interpolation] \label{prop_Inn}
Operator $\Inn{\mesh}:\;H(\div,\Omega;\SS)\cap L_q(\Omega;\SS)\to P^2_{nn}(\mesh)$ is
well defined and bounded for $q>2$.
It commutes with the piecewise divergence operator,
\(
   \div_\mesh\Inn{\mesh}=\Pi^1_\mesh\div,
\)
and has the projection property
$\Inn{\mesh}\tau=\tau$ for $\tau\in P^2_{nn}(\mesh)$.
Given $0<s\le 1$, $0\le r\le 2$, any $\tau\in H^{s,r}(\div,\Omega;\SS)$ satisfies
\begin{alignat*}{2}
   \|\tau-\Inn{\mesh}\tau\|
   &\lesssim \|h_\mesh\|_{\infty}^s \|\tau\|_s +  \|h_\mesh\|_{\infty} \|\div\tau\|\quad
   &&\forall \tau\in H^{s,0}(\div,\Omega;\SS),\\
   \|\div(\tau-\Inn{\mesh}\tau)\|_\mesh
   &\lesssim \|h_\mesh\|_\infty^r \|\div\tau\|_r
   \quad&&\forall \tau\in H^{s,r}(\div,\Omega;\SS)
\end{alignat*}
with intrinsic constants independent of $\tau$ and $\mesh$.
\end{prop}

\begin{proof}
The proof is almost verbatim to the proof of the two-dimensional case
\cite[Proposition~17]{CarstensenH_NNC}.
We only have to verify that the face degrees of freedom are bounded functionals
for $\tau\in L_q(\Omega;\SS)$ if $q>2$, and
to recall that $H^s(\Omega;\SS)\subset L_q(\Omega;\SS)$ for $0<s\le 1$ and $q=6/(3-2s)>2$
by the Sobolev embedding theorem~\cite[Theorem~7.57]{Adams}.
The former fact can be proved analogously to \cite[Lemma~4.7]{AmroucheBDG_98_VPT}.
Essential step in the proof of that lemma is the continuity
of the extension by zero from the Sobolev space $W^{1-1/q',q'}(E)$ on an edge $E$ of
a face $F$ to $W^{1-1/q',q'}(\partial F)$ where $1/q'+1/q=1$.
This zero extension is also continuous as a mapping from $W^{1-1/q',q'}(F)$
to $W^{1-1/q',q'}(\partial\el)$ where $\el\in\mesh$ and $F\in\cF(\el)$,
cf.~\cite[Corollary~1.4.4.5]{Grisvard_85_EPN}, and thus gives the result.
\end{proof}

The next result can be shown analogously to \cite[Lemma~20]{CarstensenH_NNC}.

\begin{lemma} \label{la_trdevdiv2}
Any $\tau\in P^2_{nn}(\mesh)$ with
\[
   \div_\mesh\tau=0,\quad
   \dual{\trnn(w)}{\tau}_\cS=0\ \forall w\in P^1(\mesh;\Rt)\cap H^1_0(\Omega;\Rt),
   \quad\text{and}\quad
   \vdual{\tr(\tau)}{1}=0
\]
satisfies
\[
   \|\tr(\tau)\| \le C \|\dev\tau\|
\]
with a constant $C>0$ that is independent of $\tau$ and $\mesh$.
\end{lemma}

\subsection{Finite element scheme and locking-free convergence} \label{sec_fem_conv}

Let $S^1_0(\mesh;\Rt):=P^1(\mesh;\Rt)\cap H^1_0(\Omega;\Rt)$ be
the space of continuous, piecewise linear vector-valued functions that vanish on $\Gamma$.
We select the discrete trace space
\(
   S^1_0(\cS) := \trnn\bigl(S^1_0(\mesh;\Rt)\bigr).
\)
Given $f\in L_2(\Omega;\Rt)$, our mixed finite element scheme seeks
\emph{$\sigma_h\in P^2_{nn}(\mesh)$, $u_h\in P^1(\mesh;\Rt)$, and $\eta_h\in S^1_0(\cS)$
such that}
\begin{subequations} \label{FEM}
\begin{alignat}{3}
   &\vdual{\AA\sigma_h}{\dsigma} + \vdual{u_h}{\div\dsigma}_\mesh - \dual{\eta_h}{\dsigma}_\cS
   &&= 0,\label{FEMa}\\
   &\vdual{\div\sigma_h}{\du}_\mesh - \dual{\deta}{\sigma_h}_\cS
   &&= -\vdual{f}{\du}\label{FEMb}
\end{alignat}
\end{subequations}
\emph{for any $\dsigma\in P^2_{nn}(\mesh)$, $\du\in P^1(\mesh;\Rt)$, and $\deta\in S^1_0(\cS)$.}

This scheme is locking free and converges quasi-optimally.

\begin{theorem}[locking-free convergence] \label{thm_Cea}
For given $f\in L_2(\Omega;\Rt)$ let $(\sigma,u,\eta)$ denote the solution to \eqref{VF}.
Mixed scheme \eqref{FEM} has a unique solution $(\sigma_h,u_h,\eta_h)$. It satisfies
\begin{align*}
   &\|\sigma-\sigma_h\|_{\div,\mesh} + \|u-u_h\| + \|\eta-\eta_h\|_{1/2,t,\cS}\nonumber\\
   &\lesssim
   \inf\bigl\{
   \|\sigma-\tau\|_{\div,\mesh} + \|u-v\| + \|\eta-\rho\|_{1/2,t,\cS};\;
   \tau\in P^2_{nn}(\mesh), v\in P^1(\mesh;\Rt), \rho\in S^1_0(\cS) \bigr\}
\end{align*}
with a hidden constant that is independent of $\mesh$ and $\lambda\in (0,\infty)$.

Let us furthermore assume that $u\in H^{1+s}(\Omega;\Rt)$ and $f\in H^r(\Omega;\Rt)$
for fixed $0<s\le 1$, $s\le r\le 2$, and
that there is a regularity shift $0<\sreg\le 1$ with
\[
   \|(\div\cC\varepsilon)^{-1}g\|_{1+\sreg}\le c\|g\|\quad\forall g\in L_2(\Omega;\Rt)
\]
(inversion subject to the homogeneous Dirichlet condition) for a constant $c>0$
that is independent of $g$ and $\lambda\in (0,\infty)$.
Then the estimates
\begin{align} \label{err_glob}
   \|\sigma-\sigma_h\|_{\div,\mesh} + \|u-u_h\| + \|\eta-\eta_h\|_{1/2,t,\cS}
   &\lesssim
   \|h_\mesh\|_{\infty}^s (\|u\|_{1+s} + \|\sigma\|_s + \|f\|_s),\\
   \|\div(\sigma-\sigma_h)\|_\mesh &\lesssim \|h_\mesh\|_{\infty}^r \|f\|_r, \nonumber\\
   \|u-u_h\| &\lesssim
   \|h_\mesh\|_\infty^{s+\sreg} \bigl(\|u\|_{1+s} + \|f\|_s\bigr) \nonumber
\end{align}
hold with intrinsic constants that are independent of $\mesh$ and $\lambda\in(0,\infty)$.
\end{theorem}

\begin{proof}
The proof is analogous to the proofs of Theorems~21 (quasi-optimal convergence)
and~22 (approximation orders) in \cite{CarstensenH_NNC}.
To this end we note that, for homogeneous Dirichlet data,
the condition $u_D\in H^{1+s}(\Omega;\Rt)$ with $s>0$ for a Dirichlet extension
$u_D$ in \cite[Theorem~21]{CarstensenH_NNC} is obsolete
and that the tensors $\sigma_0=\sigma_{0,h}\in\SS$ appearing in
\cite[Theorem~22]{CarstensenH_NNC} vanish, cf.~definitions (9) and (28) there.
Specifically, Lemma~\ref{la_trdevdiv2} is critical for the robust stability of the
discrete scheme and, having established the quasi-optimal convergence,
error estimate \eqref{err_glob} is consequence of Proposition~\ref{prop_Inn}.
\end{proof}
\section{Numerical experiments} \label{sec_num}


We use $\Omega=(0,1)^3$ and consider Dirichlet boundary conditions everywhere
given by the manufactured displacement
\[
   u(x_1,x_2,x_3)
   = \begin{pmatrix}
      \sin(3x_1)\cos(3x_2)\cos(3x_3)\\ \cos(3x_1)\sin(3x_2)\cos(3x_3)\\
      \cos(3x_1)\cos(3x_3)\sin(3x_3)
     \end{pmatrix},
\]
and select $f:=-\div\sigma$ where $\sigma=\CC\strain{u}$.
We choose Young's modulus $E=1$ and different values of $\nu$.
The Lam\'e parameters are $\mu = \frac E{2(1+\nu)} = \frac 1{2(1+\nu)}$ and
$\lambda=\frac{E\nu}{(1+\nu)(1-2\nu)}=\frac{\nu}{(1+\nu)(1-2\nu)}$.

We use quasi-uniform meshes and plot the individual approximation errors
$\|u-u_h\|$ (``u''), $\|\strain{u}-\strain{\wtilde u_h}\|$ (``strain''),
$\|\sigma-\sigma_h\|$ (``sigma''), and $\|\div(\sigma-\sigma_h)\|_\mesh$ (``div sigma''),
normalized by the norm of the exact solution
$\bigl(\|u\|_1^2+\|\sigma\|_{\div}^2\bigr)^{1/2}$.
Here, $\wtilde u_h\in S^1_0(\mesh;\Rt)$ is the piecewise linear interpolation of the continuous,
piecewise linear trace approximation $\eta_h$.
Note that $\|\strain{u}-\strain{\wtilde u_h}\|$ is, up to the
$L_2$ part $\|u-\wtilde u_h\|$, an upper bound of $\|\trnn(u)-\eta_h\|_{1/2,t,\cS}$.
The right-hand side functional $\vdual{f}{\cdot}$ is approximated by a 4-point Gauss
formula (integrating quadratic polynomials exactly),
and the errors are calculated by a 14-point Gauss formula (an overkill to be close to
the exact errors), cf.~\cite{JaskowiezS_21_HOS}. The non-homogeneous Dirichlet boundary condition
is approximated by assigning the corresponding vertex values to
trace approximation $\eta_h$.

Figure~\ref{fig_a} shows the results for $\nu=0.3$, and confirms the
approximation orders with respect to $h:=\|h_\mesh\|_\infty$ proved by Theorem~\ref{thm_Cea}:
$O(h)$ for the stress and trace approximations, and $O(h^2)$ for the approximations
of the displacement and the divergence of the stress.

The sequence of results for $\nu=0.3,0.45,0.49,0.4999$ shown in Figure~\ref{fig}
confirms that our scheme is locking free.
(The fact that some relative errors are initially smaller for larger values of $\lambda$
does not contradict this.)
Note that $\tr(\sigma)=\tr(\CC\strain{u})=9(2\mu+3\lambda)\cos(3x_1)\cos(3x_2)\cos(3x_3)$ and
\[
   \dev\sigma
   =-6\mu\begin{pmatrix} 0 & \sin(3x_1)\sin(3x_2)\cos(3x_3) & \sin(3x_1)\cos(3x_2)\sin(3x_3)\\
                         * & 0 & \cos(3x_1)\sin(3x_2)\sin(3x_3)\\
                         * & * & 0
         \end{pmatrix},
\]
that is,
$\|\tr(\sigma)\|$ and $\vdual{\tr(\sigma)}{1}$ are unbounded as $\nu\to 1/2$
while $\dev\sigma$ is bounded.

\begin{figure}[hb]
\centering
\begin{subfigure}{0.5\textwidth}
\includegraphics[width=\textwidth]{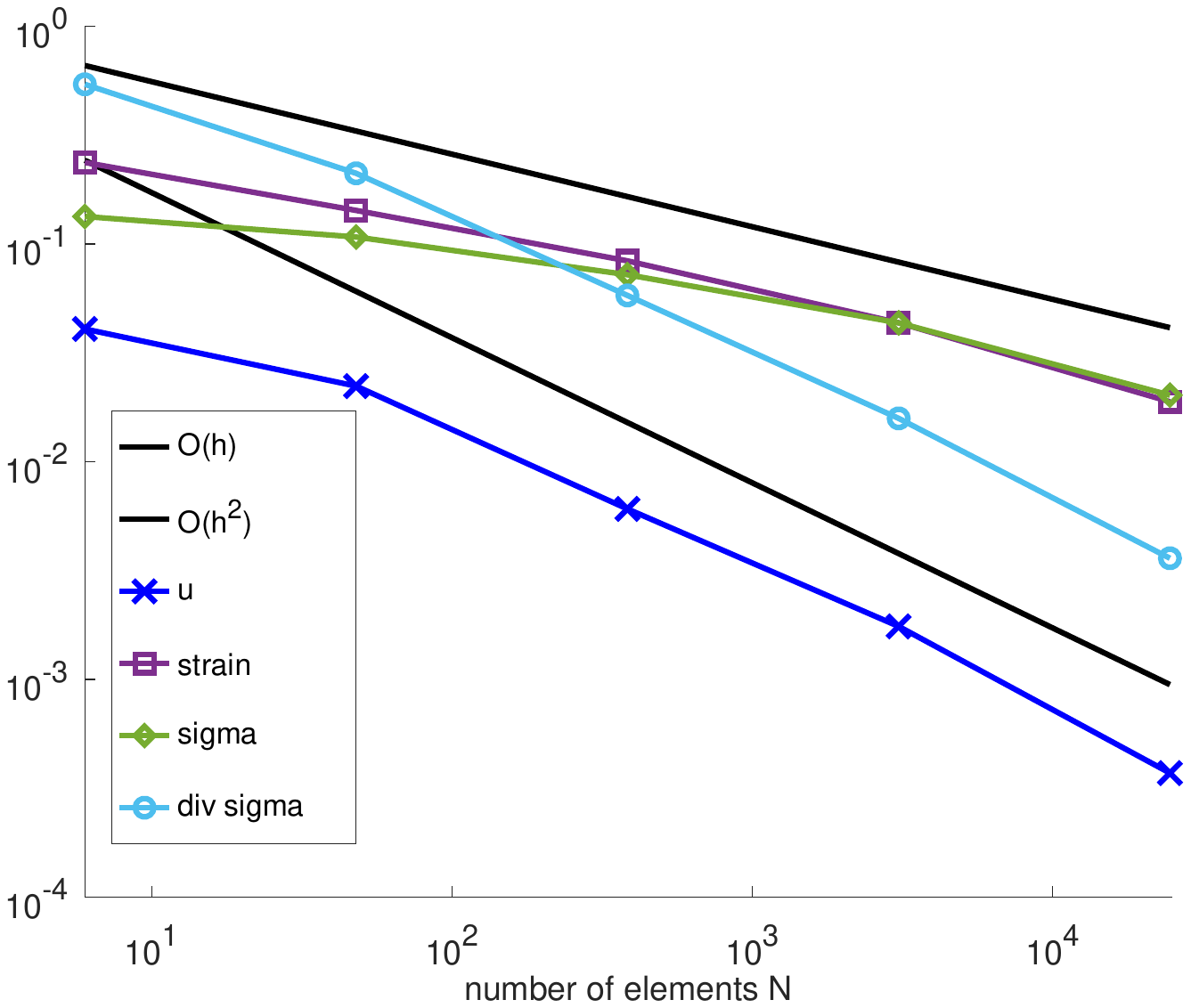}
\caption{$\nu=0.3$}
\label{fig_a}
\end{subfigure}
\hspace{-2em}
\begin{subfigure}{0.5\textwidth}
\includegraphics[width=\textwidth]{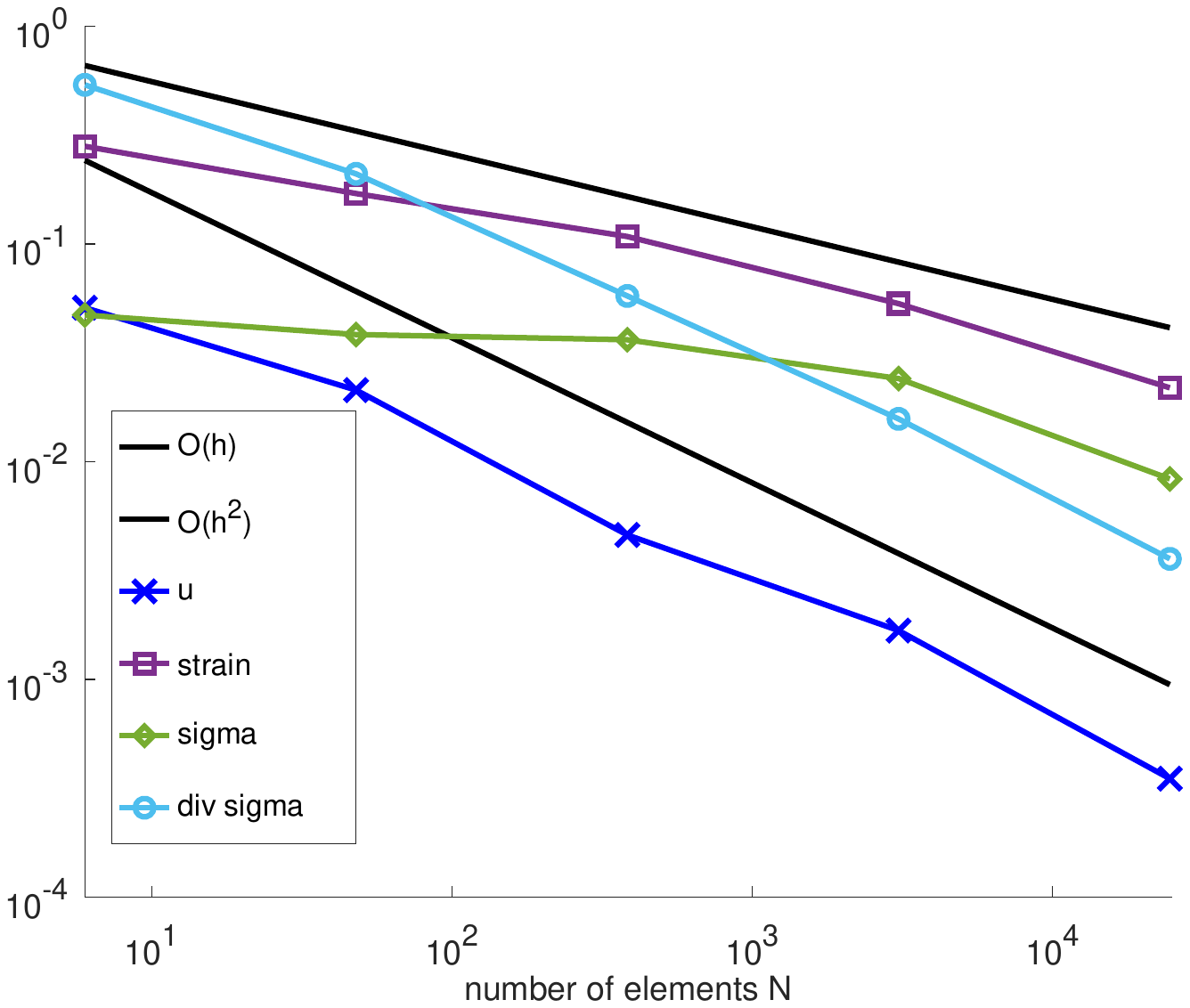}
\caption{$\nu=0.45$}
\label{fig_b}
\end{subfigure}

\begin{subfigure}{0.5\textwidth}
\includegraphics[width=\textwidth]{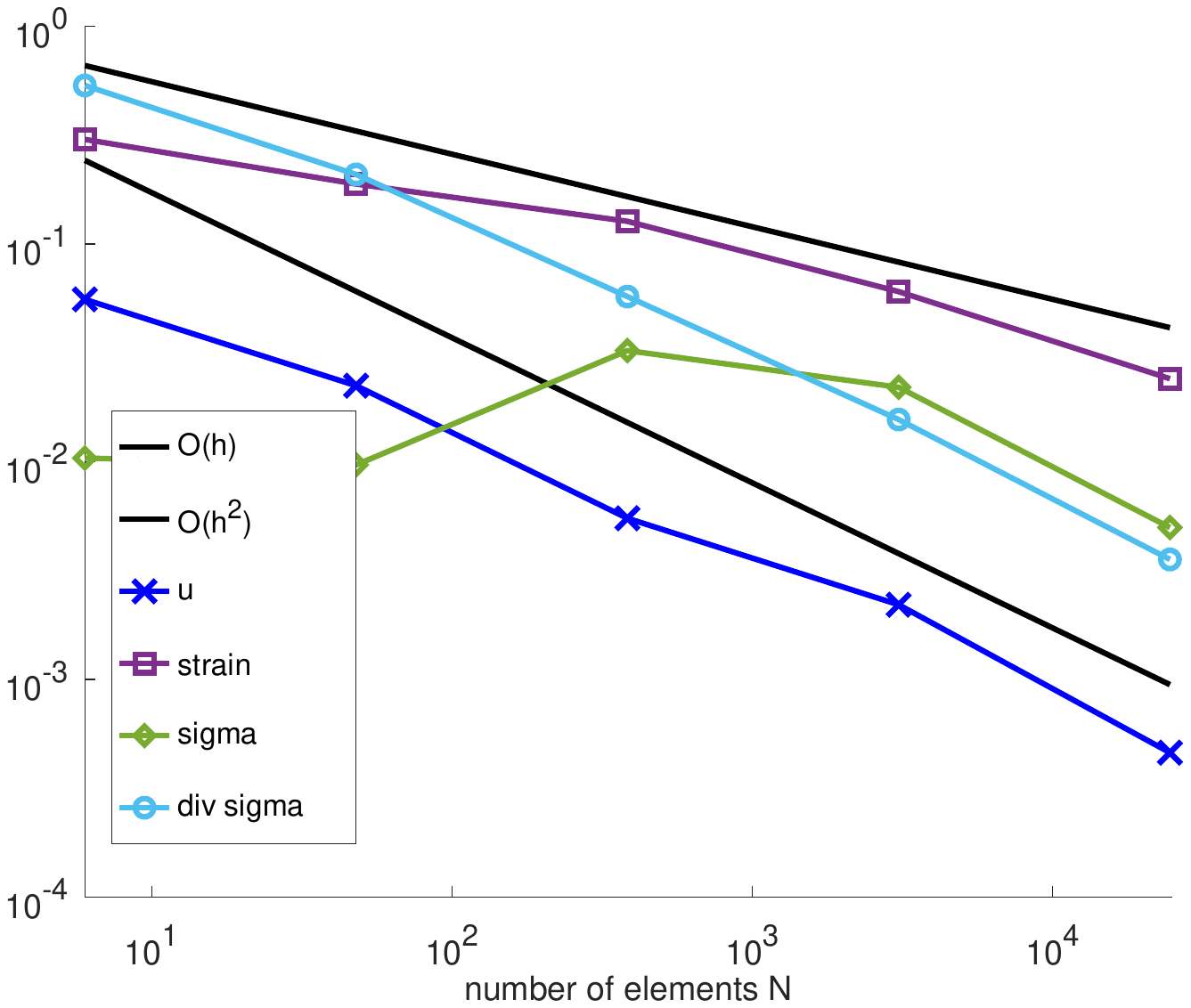}
\caption{$\nu=0.49$}
\label{fig_c}
\end{subfigure}
\hspace{-2em}
\begin{subfigure}{0.5\textwidth}
\includegraphics[width=\textwidth]{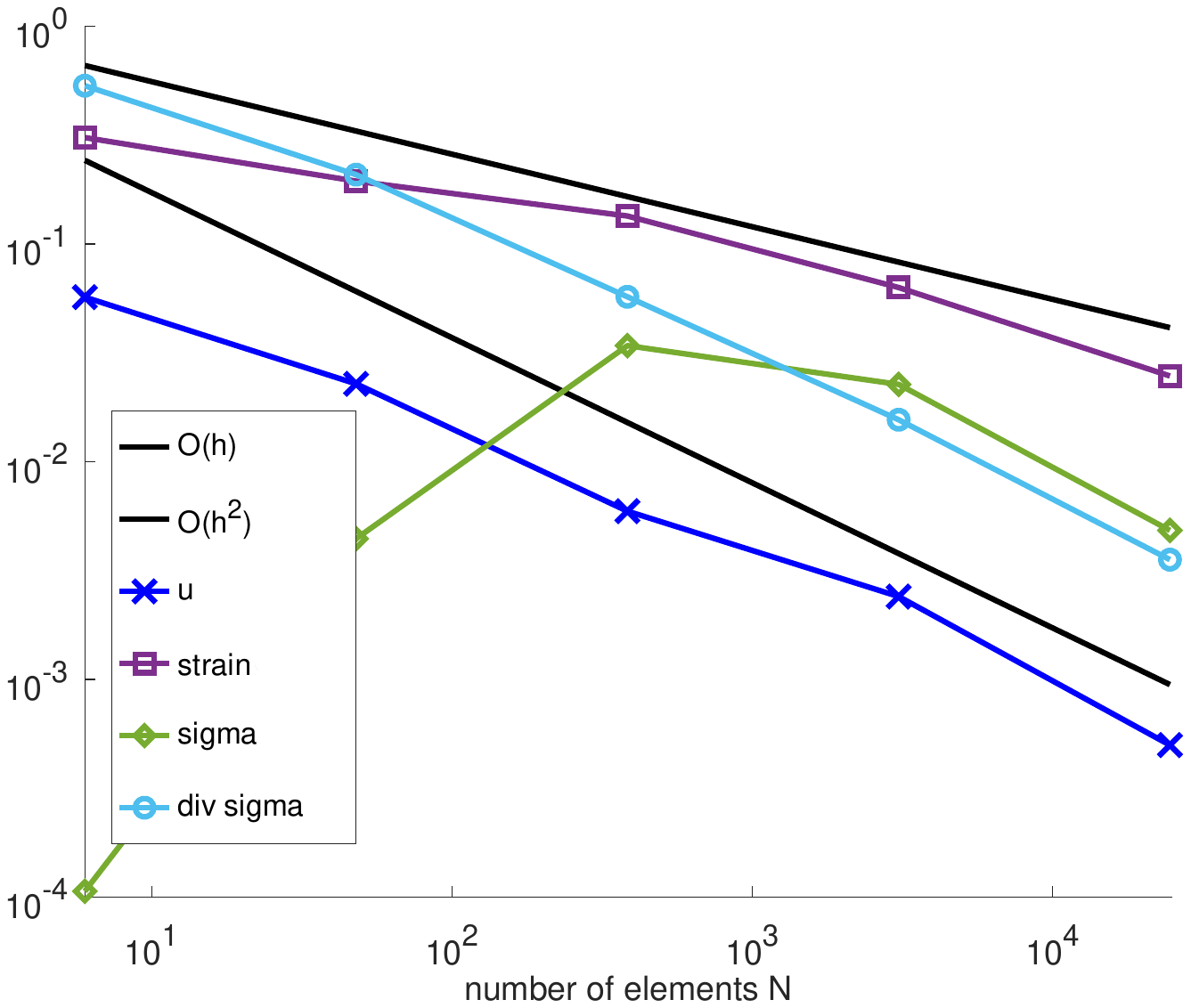}
\caption{$\nu=0.4999$}
\label{fig_d}
\end{subfigure}

\caption{Relative errors for the manufactured example with different values of $\nu$.}
\label{fig}
\end{figure}

\clearpage
\bibliographystyle{plain}
\bibliography{/home/norbert/tex/bib/bib,/home/norbert/tex/bib/heuer}

\end{document}